\documentclass[12pt,a4paper]{article}
\usepackage[top=25mm,bottom=25mm,left=20mm,right=20mm]{geometry}
\usepackage{amsmath}
\usepackage{amssymb}
\usepackage{makeidx}
\usepackage{amsfonts}
\usepackage{hyperref}
\usepackage{pictex}
\newtheorem{theorem}{Theorem}[section]

\newtheorem{lemma}[theorem]{Lemma}

\newtheorem{remark}[theorem]{Remark}

\newenvironment{proof}[1][Proof]{\noindent\textbf{#1.} }{\rule{0.5em}{0.5em}}

\begin{document}

\title{On the rate of convergence in the central
limit theorem\\ for hierarchical Laplacian}
\author{ Alexander Bendikov \thanks{%
Research of A.~Bendikov was supported by National Science Centre (Poland),
Grant 2015/17/B/ST1/00062} \and Wojciech Cygan \thanks{%
Research of W.~Cygan was supported by National Science Centre (Poland),
Grant 2015/17/B/ST1/00062 and by Austrian Science Fund project FWF P24028.}}
\date{}
\maketitle

\begin{abstract}
Let $(X,d)$ be a proper ultrametric space. Given a measure $m$ on $X$ and a
function $C(B)$ defined on the set of all non-singleton balls $B$ we
consider the hierarchical Laplacian $L=L_{C}$. Choosing a sequence $%
\{\varepsilon (B)\}$ of i.i.d. random variables we define the perturbed
function $C(B,\omega )$ and the perturbed hierarchical Laplacian $%
L^{\omega }=L_{C(\omega )}.$ We study the arithmetic means $\overline{%
\lambda }(\omega )$ of the $L^{\omega }$-eigenvalues. Under some mild
assumptions the normalized arithmetic means $\left( \overline{\lambda }-%
\mathbb{E}\overline{\lambda }\right) /\sigma \left( \overline{\lambda }%
\right) $ converge in law to the standard normal distribution. In this note
we study convergence in the total variation distance and estimate the rate
of convergence.
\end{abstract}

\section{Introduction}

The concept of hierarchical lattice and hierarchical distance was proposed
by F.J. Dyson in his famous paper on the phase transition for $\mathbf{1D}$
ferromagnetic model with long range interaction~\cite{Dyson2}. The notion of
the hierarchical Laplacian $L$, which is closely related to the Dyson's
model was studied in several mathematical papers \cite{Kvitchevski1}, \cite{Kvitchevski2}, 
\cite{Kvitchevski3}, \cite{Molchanov}, 
\cite{AlbeverioKarwowski}, \cite{BGPW}, \cite{BendikovKrupski} and
 \cite{BCW}. These papers contain some basic information about $L$
(the spectrum, the Markov semigroup, resolvent etc). In the case when the
state space is discrete and the hierarchical lattice satisfies some symmetry
conditions (homogenuity, self-similarity etc) it can be identified with some
discrete infinitely generated Abelian group $G$ equipped with a translation
invariant ultrametric $d$ and with a Haar measure $m$. The Markov semigroup $%
P^{t}=\exp (-tL)$ acting on $L^{2}(G,m)$ becomes then symmetric, translation
invariant and isotropic. In particular, $\mathsf{Spec}(L)$ is pure point and
all eigenvalues have infinite multiplicity.

In paper \cite{BGMS} we study a class of random perturbations of
hierarchical Laplacians $L.$ Each outcome $L^{\omega },\omega \in \Omega ,$
of the perturbed hierarchical Laplacian is by itself a hierarchical
Laplacian whence its spectrum $\mathsf{Spec}(L^{\omega })$ is still pure
point (with compactly supported eigenfunctions). Using the classical average
procedure one defines the integrated density of states. Contrary to the
deterministic case it may admit a continuous density w.r.t. $m$, the density
of states. The density of states detects the spectral bifurcation from the
pure point spectrum to the continuous one. The eigenvalues form locally a
Poisson point process with intensity given by the density of states. The
normalized sequence of arithmetic means of $L^{\omega }$-eigenvalues
converges in law to the standard normal distribution. In this note we study
the convergence in relative entropy, in particular in the total variation
distance. Under certain mild assumptions we establish the rate of
convergence.

\section{Preliminaries}

\label{Prem_sec}

\paragraph{Hierarchical lattice.}

Let $(X,d)$ be a proper non-compact ultrametric space. 
Recall that \emph{proper} metric space means that all closed balls are
compact, and \emph{ultrametric} $d$ is a metric which is an ultrametric,
that is 
\begin{equation*}
d(x,y)\leq \max \{d(x,z),d(z,y)\}.
\end{equation*}%
A basic consequence is that any two balls are either disjoint or one is
contained in the other. The collection of all balls with a fixed positive
radius forms a countable partition of $X$, and decreasing the radius leads
to a refined partition. This is consistent with the structure of
\textquotedblleft Hierarchical lattice\textquotedblright\ as in the old papers,
going back to \cite{Dyson2}.

Let $m$ be a Radon measure on $X$ such that $m(X)=\infty $ and $m(B)>0$ for
each closed ball which is not a singleton, and $m(\{a\})>0$ if and only if $%
a $ is an isolated point of $X$. Let $\mathcal{B}$ be the collection of all
balls with $m(B)>0$. Each $B\in \mathcal{B}$ has a unique \emph{predecessor}
or \emph{parent} $B^{\prime }\in \mathcal{B}\setminus \{B\}$ which contains $%
B$ and is such that $B\subseteq D\subseteq B^{\prime }$ for $D\in \mathcal{B}
$ implies $D\in \{B,B^{\prime }\}$. In this case, $B$ is called a \emph{%
successor} of $B^{\prime }$. Since $X$ is proper, each non-singleton ball
has only finitely many (and at least 2) successors. Their number is the 
\emph{degree} of the ball.

\paragraph{Hierarchical Laplacian.}

We consider a function $C:\mathcal{B}\rightarrow (0\,,\,\infty )$ which
satisfies, for all $B\in \mathcal{B}$ and all non-isolated $a\in X$, 
\begin{equation}
\begin{aligned} \lambda(B)&=\sum\limits_{D\in\mathcal{B}\,:\,D\supseteq
B}C(D)<\infty\,, \quad\text{and} \\
\lambda(\{a\})&=\sum_{B\in\mathcal{B}\,:\,B \ni a}C(B)=\infty\,.
\end{aligned}  \label{eq:C1condition}
\end{equation}%
Let $\mathcal{F}$ be the set of all locally constant functions having
compact support. It is known that $\mathcal{F}$ consists of continuous
functions and is dense in all $L^{p}(X,m).$ Given the space $X$, the measure 
$m$ and the function $C:\mathcal{B}\rightarrow (0\,,\,\infty )$, we define
(pointwise) the hierarchical Laplacian $L_{C}\,$: for each $f$ in $\mathcal{F%
}$ and $x\in X$ we set 
\begin{equation*}
L_{C}f(x):=\sum\limits_{B\in \mathcal{B}\,:\,B\ni x}C(B)\left( f(x)-\frac{1}{%
m(B)}\int_{B}f\,dm\right) .  
\end{equation*}%
The operator $(L_{C},\mathcal{F})$ acts in $L^{2}(X, m)$, is symmetric
and admits a complete system of eigenfunctions $\{f_{B}:B\in \mathcal{B}\}$ 
given by 
\begin{equation*}
f_{B}=\frac{\mathbf{1}_{B}}{m(B)}-\frac{\mathbf{1}_{B^{\prime }}}{%
m(B^{\prime })}\,.  
\end{equation*}%
The eigenvalue
corresponding to $f_{B}$ depends only on $B^{\prime }$ and is $\lambda
(B^{\prime })$, as given in \eqref{eq:C1condition}. Since all $f_{B}$ belong
to $\mathcal{F}$ and the system $\{f_{B}:B\in \mathcal{B}\}$ is complete we
conclude that $(L_{C},\mathcal{F})$ is an essentially self-adjoint operator.
By a slight abuse of notation, we shall write $(L_{C},\mathsf{Dom}_{L_{C}})$
for its unique self-adjoint extension. For all of this we refer to \cite%
{BendikovKrupski}, \cite{BGP} and \cite{BGPW}. 

\paragraph{Homogeneous hierarchical Laplacian.}\label{sec:homogeneous}

For the analysis undertaken in this paper, we require that the ultrametric
measure space $(X,d,m)$ and the hierarchical Laplacian $L_{C}$ are \emph{%
homogeneous,} that is there exists a group 
of isometries of $(X,d)$ which

\begin{itemize}
\item acts transitively on $X$, and

\item leaves both the reference measure $m$ and the function $C(B)$
invariant.
\end{itemize}

The first assumption implies that $(X,d)$ is either discrete or perfect.
Basic examples which we have in mind are

\begin{enumerate}
\item $X=\mathbb{Q}_{p}$ -- the ring of $p$-adic numbers, where $p \ge 2$
(integer).
\item $X = 
\bigoplus_{j=1}^{\infty} \mathbb{Z}/p_j\,\mathbb{Z}$ -- the direct sum of
countably many cyclic groups.
\item $X=S_{\infty}\,$ -- the infinite symmetric group, that is, the group
of all permutations of the positive integers that fix all but finitely many
elements.
\end{enumerate}

The homogeneity assumptions and the fact that $X$ is non-compact imply that
we have the following two cases.

\begin{itemize}
\item[Case 1.] $\;X$ is perfect, and $\{d(x,y):y\in X\}=\{0\}\cup
\{r_{k}:k\in \mathbb{Z}\}$, where $r_{k}<r_{k+1}$ with $\lim\limits_{k%
\rightarrow \infty }r_{k}=\infty $ and $\lim\limits_{k\rightarrow -\infty
}r_{k}=0\,$;

\item[Case 2.] $\;X$ is countable, and $\{d(x,y):y\in X\}=\{r_{k}:k\in 
\mathbb{N}_{0}\}$, where $r_{0}=0$, $r_{k}<r_{k+1}$ with $%
\lim\limits_{k\rightarrow \infty }r_{k}=\infty \,$.
\end{itemize}

In both cases, we let $\mathcal{B}_{k}$ be the collection of all closed
balls of diameter $r_{k}$. This is a partition of $X$, and it is finer than $%
\mathcal{B}_{k+1}$. By homogeneity, all balls in $\mathcal{B}_{k}$ are
isometric. In particular, the number $n_{k}$ of successor balls is the same
for each ball in $\mathcal{B}_{k}\,$, where $k\in \mathbb{Z}$ in Case 1, and 
$k\in \mathbb{N}_{0}$ in Case 2. We notice that the degree sequence $(n_{k})$
satisfies $2\leq n_{k}<\infty $.

It is useful to associate an infinite \emph{tree} with $X$, see Figure 1. Its vertex set
is $\mathcal{B}$, and there is an edge between any $B\in \mathcal{B}$ and
its predecessor $B^{\prime }$. In this situation, $\mathcal{B}_{k}$ is the 
\emph{horocyle} $H_{k}$ of the tree with index $k$, and $X$ is the (lower) \emph{%
boundary} of that tree. For more details see \cite{BGPW}, and \cite%
{Figa-Tal1}, \cite{Figa-Tal2}.

For having homogeneity, the reference measure $m$ is also uniquely defined
up to a constant factor. If we set $m(B)=1$, for each $B\in \mathcal{B}_{0}$%
, then, for any $k\in \mathbb{Z}$ (Case 1), resp. $k\in \mathbb{N}_{0}$
(Case 2), and for $B\in \mathcal{B}_{k}$, 
\begin{equation*}
m(B)=%
\begin{cases}
n_{1}n_{2}\cdots n_{k} & \text{for $k>0$, and} \\ 
1/(n_{k+1}n_{k+2}\cdots n_{0}) & \text{for $k<0$ in Case 1.}%
\end{cases}%
\end{equation*}%
This determines $m$ uniquely as a measure on the Borel $\sigma $-algebra of $%
X$. Regarding the hierarchical Laplacian, homogeneity means that $C(B)=C_{k}$
is the same for each $B\in \mathcal{B}_{k}\,$. Along with the function $C(B)$%
, also the eigenvalues of \eqref{eq:C1condition} depend only on $k$: 
\begin{equation*}
\lambda (B)=\lambda _{k}\;\text{ for all }\;B\in \mathcal{B}_{k}\,,\quad 
\text{where}\quad \lambda _{k}=\sum_{\ell \geq k}C_{\ell }\,.
\end{equation*}

As noticed in \cite{Figa-Tal1}, \cite{Figa-Tal2}, the homogeneous
ultrametric measure space $(X,m)$ can then be identified with a locally
compact totally disconnected group $G$ equipped with its Haar measure. In
fact, we may even identify it with an Abelian group. If $(r_{k})$ is the
sequence of distances defied above then $G_{k}=B(e,r_{k})$ is a compact-open
subgroup of $G\equiv X$, 
\begin{equation*}
G=\bigcup_{k}G_{k}\,,\quad \text{and}\quad n_{k}=[G_{k}:G_{k-1}]
\end{equation*}%
gives the degree sequence. The collection $\mathcal{B}_{k}$ of balls with
diameter $r_{k}$ consists of the left cosets of $G_{k}$ in $G$. We usually
normalize the Haar measure $m$ such that $m(G_{0})=1$.

$$
\beginpicture \label{fig}

\setcoordinatesystem units <.8mm,1.3mm>

\setplotarea x from -10 to 104, y from -14 to 46

\arrow <6pt> [.2,.67] from 2 2 to 40 40

\plot 32 32 62 2 /

 \plot 16 16 30 2 /

 \plot 48 16 34 2 /

 \plot 8 8 14 2 /

 \plot 24 8 18 2 /

 \plot 40 8 46 2 /

 \plot 56 8 50 2 /

 \plot 4 4 6 2 /

 \plot 12 4 10 2 /

 \plot 20 4 22 2 /

 \plot 28 4 26 2 /

 \plot 36 4 38 2 /

 \plot 44 4 42 2 /

 \plot 52 4 54 2 /

 \plot 60 4 58 2 /



\arrow <6pt> [.2,.67] from 99 29 to 88 40

 \plot 66 2 96 32 /

 \plot 70 2 68 4 /

 \plot 74 2 76 4 /

 \plot 78 2 72 8 /

 \plot 82 2 88 8 /

 \plot 86 2 84 4 /

 \plot 90 2 92 4 /

 \plot 94 2 80 16 /


\setdots <3pt>
\putrule from -4.8 4 to 102 4
\putrule from -4.5 8 to 102 8
\putrule from -2 16 to 102 16
\putrule from -1.7 32 to 102 32

\setdashes <2pt> \linethickness=.5pt
\putrule from -2 -5 to 102 -5

\put {$\vdots$} at 32 0
\put {$\vdots$} at 64 0

\put {$\dots$} [l] at 103 6
\put {$\dots$} [l] at 103 24

\put {$H_{2}$} [l] at -14 32
\put {$H_{1}$} [l] at -14 16
\put {$H_0$} [l] at -14 8
\put {$H_{-1}$} [l] at -14 4
\put {$G$} [l] at -14 -5
\put {$\vdots$} at -10 0
\put {$\vdots$} [B] at -10 36

\put {$\scriptstyle\bullet$} at 8 8
\put {$G_0$} [rb] at 7.2 8.8
\put {$\scriptstyle\bullet$} at 16 16
\put {$G_{1}$} [rb] at 15.2 16.8
\put {$\scriptstyle\bullet$} at 32 32
\put {$G_{2}$} [rb] at 31.2 32.8
\put {$\scriptstyle\bullet$} at 36 4 
\put {$S$} [rb] at 35.2 4.8
\put {$\scriptstyle\bullet$} at 56 8 
\put {$T$} [lb] at 56.8 8.8
\put {$\scriptstyle\bullet$} at 48 16 
\put {$S \curlywedge T$} [lb] at 48.8 16.8

\put {$\varpi$} at 42 42
\put {$\varpi$} at 86 42
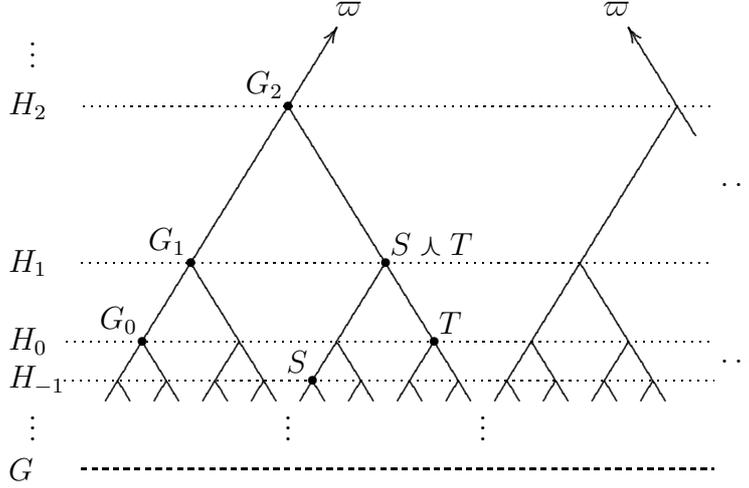
\put {\rm Figure 1. Tree of balls $\mathcal{T}(X)$ with forward degree $n_l=2$.} at 47 -11

\endpicture
$$

\paragraph{Random perturbations.}

Let $L_{C}$ be the homogeneous hierarhical Laplacian. Let $\{\varepsilon
(B)\}_{B\in \mathcal{B}}$ be a sequence of symmetric i.i.d. random variables
defined on the probability space $(\Omega ,\mathbb{P})$ and taking values in
some small interval $[-\epsilon ,\epsilon ]\subset \left( -1,1\right) .$ We
define the perturbed function $C(B,\omega )$ and the perturbed
hierarchical Laplacian as follows:%
\begin{equation*}
C(B,\omega )=C(B)(1+\varepsilon (B,\omega ))
\end{equation*}%
and%
\begin{equation*}
L^{\omega }f(x)=L_{C(\omega )}f(x)=\sum\limits_{B\in \mathcal{B}:\text{ }%
x\in B}C(B,\omega )\left( f(x)-\frac{1}{m(B)}\int_{B}f\,dm\right) .
\end{equation*}%
Evidently $L^{\omega }$ may well be non-homogeneous for some $\omega \in
\Omega $. Still it has a pure point spectrum for all $\omega $ but the
structure of the closed set $\mathsf{Spec}(L^{\omega })$ can be quite
complicated, see \cite{BendikovKrupski} for various examples.

Let us fix a horocyle $H$ and compute the eigenvalue $\lambda (B,\omega )$
for $B$ in $H$. Without loss of generality we may assume that $H=H_{0}$. Let 
$\varpi $ be the Alexandrov point and let $\{B_{k}\}_{k\geq 0}$ be the
unique infinite geodesic path in $\mathcal{T}(X)$ from $B$ to $\varpi $. We
have 
\begin{align*}
\lambda (B,\omega )& =\sum\limits_{k\geq 0}C(B_{k},\omega
)=\sum\limits_{k\geq 0}C_{k}\left( 1+\varepsilon (B_{k},\omega )\right) 
\\
& =\lambda _{0}\left( 1+\sum\limits_{k\geq 0}a_{k}\varepsilon (B_{k},\omega
)\right) =\lambda _{0}\left( 1+U(B,\omega )\right) ,
\end{align*}%
where $a_{k}=C_{k}/\lambda _{0},$ and%
\begin{equation}
U(B,\omega )=\sum\limits_{k\geq 0}a_{k}\varepsilon (B_{k},\omega ).
\label{U-epsilon}
\end{equation}%
Notice that $\sum_{k\geq 0}a_{k}=1$ and that $\{U(B)\}_{B\in H}$ are
(dependent) identically distributed symmetric random variables taking values
in some symmetric interval $I\subsetneqq \left( -1,1\right) .$ In
particular, $\mathbb{E}\lambda (B,\cdot )=\lambda _{0}$.

\paragraph{Normal approximation.}

Let us choose a reference point $o\in X,$ say the neutral element in our
group-identification, and let $\mathcal{O}$ denote the family of all balls $%
O $ centered at $o.$ Let $\mathcal{B}_{0}(O),O\in \mathcal{O},$ be the set
of all balls $B$ in $O$ each of which belongs to the horocycle $H_{0}$. We
set
\begin{equation*}
\overline{\lambda }_{O}(\omega )=\frac{1}{\left\vert \mathcal{B}%
_{0}(O)\right\vert }\sum\limits_{B\in \mathcal{B}_{0}(O)}\lambda (B,\omega )
\end{equation*}%
and\footnote{In the paper we use both $\mathrm{Var(X)}$ and $\sigma (X)^2$ to denote the variance of the random variable $X$.}  
\begin{equation*}
\Lambda _{O}\left( \omega \right) =\frac{\overline{\lambda }_{O}(\omega
)-\lambda _{0}}{\sigma \left( \overline{\lambda }_{O}\right) },
\end{equation*}%
where $\left\vert 
\mathcal{B}_{0}(O)\right\vert $ stands for cardinality of the finite set $%
\mathcal{B}_{0}(O)$.

According to \cite[Theorem 3.2]{BGMS}, as $O\rightarrow \varpi ,$ $\Lambda
_{O}$ converges in law to the standard normal random variable $Z$ whenever
the following condition holds 
\begin{equation}
1/\kappa \leq C(B)\left( \mathsf{diam}(B)\right) ^{\delta /2}\leq \kappa ,
\label{C-B condition11}
\end{equation}%
or equivalently, 
\begin{equation}
1/2\kappa \leq \lambda (B)\left( \mathsf{diam}(B)\right) ^{\delta /2}\leq
2\kappa ,  \label{lambda-B condition}
\end{equation}%
for some $\kappa >0$ and $\delta \geq 1$,

The main aim of this paper is to strengthen this result, namely we want to
prove that under certain conditions the convergence $\Lambda_O\to Z$ holds
in the total variation distance, see Theorem \ref{Thm_approx}.

\paragraph{Metric matters.}\label{sec: Dist}

Given two probability distributions $P$ and $Q$ on the real line the total
variation distance from $P$ to $Q$ is defined as 
\begin{align*}
\Vert P - Q \Vert _{TV} = 2\!\!\! \sup_{A\in \mathrm{B}(\mathbb{R})}%
\Big\vert P(A) - Q(A) \Big\vert.
\end{align*}
Notice that if $P$ and $Q$ are absolutely continuous with respect to some
measure $\theta$ and $f$ and $g$ are their densities then 
\begin{align*}
\Vert P- Q \Vert _{TV} = 2\Vert f-g \Vert_{L^1(\text{d} \theta)}.
\end{align*}
The relative entropy (also called the Kullback-Leibler distance) of $P$ with
respect to $Q$ is defined as 
\begin{align*}
D(P\! \parallel \! Q) = \int_{\mathbb{R}}\log_2 \left( \frac{\text{d} P}{%
\text{d} Q}\right) \text{d} P,
\end{align*}
when $P $ is absolutely continuous with respect to $Q$ and $D(P\!\!
\parallel \!\! Q) = +\infty$ otherwise.

Let $X$ be a random variable with the density $p$ with respect to the
Lebesgue measure. Its differential entropy is defined as 
\begin{align*}
h(X) = - \int_{\mathbb{R}}p(x)\log_2 p(x)\, \text{d} x.
\end{align*}
We also consider the following quantity 
\begin{align*}
D(X) = h(W) - h(X) = \int_{\mathbb{R}}p(x)\log_2 \frac{p(x)}{q(x)}\, \text{d}
x,
\end{align*}
where $W$ is a normal random variable with the density $q$ (with resp. to
the Lebesgue measure) such that $\mathbb{E}(W) = \mathbb{E} (X)$ and $%
\mathrm{Var}(W) = \mathrm{Var}(X)$. Observe that $D(X) = D(P_X\!\parallel \!
P_W)$, where $P_X$ denotes the distribution of the random variable $X$. We
have $D(a+bX)= D(X)$, $0\neq b,a\in \mathbb{R}$, and thus $D$ is mean and
variance invariant. Moreover, if $P_X$ is equal to $N(0,\sigma^2)$ then 
\begin{align}  \label{entropy_normal}
h(X) = \frac{\log_2(2\pi e\sigma^2)}{2}\quad \mathrm{and}\quad D(X)=0.
\end{align}

We recall that by the Pinsker inequality \cite{Pinsker}, the entropic
distance dominates the total variation, that is 
\begin{align*}  
\Vert P_X - P_W\Vert _{TV}^2 \leq 2D(X).
\end{align*}

We shall need the following bound on the relative entropy expressed in terms
of densities \cite[Lemma 2.2]{Bobkov}. Let $X$ be a random variable whose
distribution function $F_{X}(x)$ is absolutely continuous with $%
F_{X}^{\prime }(x)=$ $p(x),$ assume further that the first absolute moment
of $X$ is finite. Let $Z$ be the standard normal random variable and $%
F_{Z}^{\prime }(x)=\phi (x)$ be its density. For any $T\geq 0$ we have 
\begin{equation}  \label{Entropy_density}
\begin{split}
D(P_{X}\!\parallel \!P_{Z})& \leq e^{-T^{2}/2}+\sqrt{2\pi }%
\int_{-T}^{T}\left( p(x)-\phi (x)\right) ^{2}e^{x^{2}/2}\text{d}x \\
& \quad +\frac{1}{2}\int_{|x|\geq T}x^{2}p(x)\text{d}x+\int_{|x|\geq
T}p(x)\log_2 p(x)\text{d}x.
\end{split}%
\end{equation}

\paragraph{Berry-Essen bounds.}

Let $X_{1},X_{2},\ldots ,X_{n}$ be independent random variables with $%
\mathbb{E}(X_{k})=0$ and with finite variances $\sigma _{k}^{2}=\mathbb{E}%
(X_{k}^{2})$. Assuming that for some $s>2$ and all $1\leq k\leq n,$ $\mathbb{%
E}|X_{k}|^{s}<\infty $ we define the following quantities 
\begin{equation*}
B_{n}=\sum_{k=1}^{n}\sigma _{k}^{2}\text{ ,\ \ \ \ }S_{n}=\frac{X_{1}+\ldots
+X_{n}}{B_{n}^{1/2}},
\end{equation*}%
and 
\begin{equation*}
\mathbb{L}_{s}=\frac{1}{B_{n}^{s/2}}\sum_{k=1}^{n}\mathbb{E}|X_{k}|^{s},
\end{equation*}%
the Lyapunov ratios. Let $P_{n}$ be the distribution of the random variable $%
S_{n}$ and $\mathcal{N}$ be the standard normal distribution. One of the
main ingredients in our analysis is the following result, see {\cite[%
Theorems 1.1 and 1.2]{Bobkov}}.

\begin{theorem}
\label{THM_Berry-Essen} Assume that $D(X_{k})\leq A$ for some positive
constant $A$, and all $1\leq k\leq n.$ The following statements hold true

\begin{enumerate}
\item[1.] $\mathbb{L}_{3}<+\infty $ implies $\Vert P_{n}-\mathcal{N}\Vert
_{TV}\leq C\mathbb{L}_{3},$

\item[2.] $\mathbb{L}_{4}<+\infty $ implies $D(S_{n})\leq C^{\prime }\mathbb{%
L}_{4},$
\end{enumerate}
where the constants $C$ and $C^{\prime }$ depend only on $A$.
\end{theorem}

We shall need the following technical result which we will prove here for completeness.

\begin{lemma}
\label{lemma_entropy} Let $X,Y$ be two independent random variables such that%
$\mathrm{Var}(X)+\mathrm{Var}(Y)=1$. Then 
\begin{equation*}
D(X+Y)\leq \mathrm{Var}(X)D(X)+\mathrm{Var}(Y)D(Y).
\end{equation*}%
In particular, for any collection of independent random variables $\{X_{k}\}$%
, 
\begin{equation*}
D\left( \sum_{k=1}^{n}X_{k}\right) \leq \max_{1\leq k\leq n}D(X_{k}).
\end{equation*}
\end{lemma}

\begin{proof}
We apply the following inequality, see \cite[Theorem D.1]{Johnson}, 
\begin{equation}  \label{power_entropy}
2^{2h(X)}+2^{2h(Y)}\leq 2^{2h(X+Y)}.
\end{equation}%
According to \eqref{entropy_normal}, we have 
\begin{equation*}
D(X+Y)=h(Z)-h(X+Y)\quad \mathrm{and}\quad h(Z)=\log _{2}\sqrt{2\pi e}.
\end{equation*}%
Similarly, we write 
\begin{align*}
D(X) = h(Z_1)+ h(X)\quad \mathrm{and}\quad D(Y) = h(Z_2)+ h(Y),
\end{align*}
where $Z_{1}$ and $Z_{2}$ are normal random variables such that $\mathrm{Var}%
(Z_{1})=\mathrm{Var}(X)$ and $\mathrm{Var}(Z_{2})=\mathrm{Var}(Y)$. Using
this in \eqref{power_entropy} we get 
\begin{equation*}
2^{h(Z_{1})}\cdot 2^{-2D(X)}+2^{h(Z_{2})}\cdot 2^{-2D(Y)}\leq
2^{2h(Z_{1}+Z_{2})}\cdot 2^{-2D(X+Y)},
\end{equation*}%
whence 
\begin{equation*}
\mathrm{Var}(X)\cdot 2^{-2D(X)}+\mathrm{Var}(Y)\cdot 2^{-2D(Y)}\leq
2^{-2D(X+Y)}.
\end{equation*}%
Since $x\mapsto 2^{-2x}$ is convex, we get 
\begin{equation*}
2^{-2\left[ \mathrm{Var}(X)D(X)+\mathrm{Var}(Y)D(Y)\right] }\leq \mathrm{Var}%
(X)\cdot 2^{-2D(X)}+\mathrm{Var}(Y)\cdot 2^{-2D(Y)}.
\end{equation*}%
The result follows.
\end{proof}

\section{Central limit theorem}\label{sec: Normal}

The main result of the article is the following theorem.

\begin{theorem}
\label{Thm_approx} In notation of Section \ref{Prem_sec}, suppose that
condition \eqref{C-B condition11} holds with $\delta \geq 1$. Assume that
the random variable $\epsilon (B)$ has a bounded density with respect to the
Lebesgue measure. Denote $v_{N}=n_{1}n_{2}\cdot \cdot \cdot n_{N},$ clearly $%
v_{N}\geq 2^{N}$. For $O\in H_{N}$ the following inequality holds 
\begin{equation*}
\Vert P_{\Lambda _{O}}-\mathcal{N}\Vert _{TV}\leq 
\begin{cases}
C\,N^{-\frac{1}{2}} & \text{for $\delta =1,$} \\ 
C\,v_{N}^{-\frac{3}{2}\min \{\delta -1,\frac{1}{3}\}} & \text{for $\delta >1$%
},%
\end{cases}%
\end{equation*}%
for some $C>0$ and all $N\geq 1$. In particular, independently on the group
structure 
\begin{equation*}
\Vert P_{\Lambda _{O}}-\mathcal{N}\Vert _{TV}\leq 
\begin{cases}
C\,N^{-\frac{1}{2}} & \text{for $\delta =1,$} \\ 
C\,2^{-\frac{3}{2}N\min \{\delta -1,\frac{1}{3}\}} & \text{for $\delta >1$}.%
\end{cases}%
\end{equation*}
\end{theorem}

\begin{proof}
For $O\in H_{N}$ we write $\overline{\lambda }_{N}(\omega )$ for $\overline{%
\lambda }_{O}(\omega )$ and $\Lambda _{N}(\omega )$ for $\Lambda _{O}(\omega )$%
, then 
\begin{align}
\overline{\lambda }_{N}(\omega )&=\lambda _{0}\left( 1+\overline{U}
_{N}(\omega )\right) ,  \notag \\
\overline{U}_{N}(\omega )&=\frac{1}{v_{N}}\sum\limits_{B\in \mathcal{B}%
_0(O)}U(B,\omega ) ,  \label{U-bar} \\
\Lambda _{N}(\omega )&=\frac{\overline{U}_{N}(\omega )}{\sigma \left(
\overline{U}_{N}\right) }. \notag
\end{align}
To estimate the quantity $\Vert P_{\Lambda _{O}}-\mathcal{N}\Vert _{TV}$ we
apply Theorem \ref{THM_Berry-Essen}. We distinguish two cases \textbf{$%
(\delta >1)$ }and \textbf{$(\delta =1).$}\newline
\textbf{The case $(\delta >1)$}: Let $\{O_{k}\}_{k\geq N}$ be the infinite
geodesic path from $O$ to $\varpi $. Similarly, for each $B\in H_0$ we pick
the infinite geodesic path $\{B_{k}\}_{k\geq 0}$ from $B$ to $\varpi $.
Applying $(\ref{U-epsilon})$ to equation $(\ref{U-bar})$ we obtain%
\begin{eqnarray*}
\overline{U}_{N} &=&\frac{1}{n_{1}...n_{N}}\sum\limits_{B\in H_0:B\subset
O}\sum\limits_{k\geq 0}a_{k}\varepsilon (B_{k})=\frac{1}{n_{1}...n_{N}}%
\sum\limits_{k\geq 0}a_{k}\sum\limits_{B\in H_0:B\subset O}\varepsilon
(B_{k}) \\
&=&\frac{1}{n_{1}...n_{N}}\left( a_{0}\sum\limits_{B_{0}\in
H_{0}:B_{0}\subseteq O}\varepsilon (B_{0})+a_{1}n_{1}\sum\limits_{B_{1}\in
H_{1}:B_{1}\subseteq O}\varepsilon (B_{1})\right. \\
&&\left. +a_{2}n_{1}n_{2}\sum\limits_{B_{2}\in H_{2}:B_{2}\subseteq
O}\varepsilon (B_{2})+...+a_{N}n_{1}n_{2}...n_{N}\varepsilon (O_{N})\right)
\\
&&\quad +a_{N+1}\varepsilon (O_{N+1})+a_{N+2}\varepsilon (O_{N+2})+...\text{ 
}.
\end{eqnarray*}%
Let us introduce two random variables 
\begin{align}
\mathcal{U}_{N}=\frac{a_{0}}{n_{1}...n_{N}}\sum\limits_{B_{0}\in
H_{0}:B_{0}\subseteq O}\varepsilon (B_{0})+\frac{a_{1}}{n_{2}...n_{N}}%
\sum\limits_{B_{1}\in H_{1}:B_{1}\subseteq O}\varepsilon
(B_{1})+\ldots
 + a_{N}\varepsilon (O_{N})  \label{U-L equation}
\end{align}
and%
\begin{equation}\label{V-L equation}
X_{N}=a_{N+1}\varepsilon (O_{N+1})+a_{N+2}\varepsilon (O_{N+2})+...\text{ }.
\end{equation}%
Random variables $\mathcal{U}_{N}$ and $X_{N}$ are independent, have zero
mean and%
\begin{equation*}
\overline{U}_{N}=\mathcal{U}_{N}+X_{N}.  
\end{equation*}%
Let us denote by $\mathcal{B}_{k}(O)$ the set of all balls $B\subset O$
which belong to the horocycle $H_{k}$, and let $\mathcal{B}^{\ast
}(O)=\bigcup_{k=0}^{N}\mathcal{B}_{k}(O)$. Then 
\begin{equation*}
\Lambda _{N}=\frac{\sum_{B\in \mathcal{B}^{\ast }(O)}X_{B}+X_{N}}{B_{N}^{1/2}%
},
\end{equation*}%
where 
\begin{equation*}
X_{B}=\frac{a_{k}}{n_{k+1}...n_{N}}\varepsilon (B),\text{ for }B\in \mathcal{%
B}_{k}(O),
\end{equation*}%
and 
\begin{equation*}
B_{N}=\sum_{B\in \mathcal{B}^{\ast }(O)}\sigma ^{2}\left( X_{B}\right) +\sigma
^{2}\left( X_{N}\right) .
\end{equation*}

As $\epsilon (B)$ are bounded i.i.d. having bounded density, $D(X_{B}) =
D(\epsilon (B))<\infty$ for all balls $B\in \mathcal{B}^{\ast }(O)$. We
claim that, for some $A>0$, 
\begin{equation*}
D(X_{N})\leq A,\quad \mathrm{for\ all}\ N.
\end{equation*}%
Indeed, let $\widetilde{X_{N}}=X_{N}/\sqrt{\sigma (X_{N})}$. Then $%
D\left( X_{N}\right) =D\big( \widetilde{X_{N}}\big)$. Random variable $\widetilde{X_{N}}$ is
bounded because, by \eqref{C-B condition11}, 
\begin{equation*}
\left\vert \widetilde{X_{N}}\right\vert \leq \frac{\sum_{k>N}a_{k}}{\sqrt{\sum_{k>N}a_{k}^{2}}}%
\leq C \frac{a_{N+1}}{\sqrt{a_{N+1}^{2}}}=C<\infty,
\end{equation*}%
for some $C>0$. It follows that the density $\widetilde{p_{N}}$ of $%
\widetilde{X_{N}}$ is continuous and compactly supported function, whence
applying \eqref{Entropy_density} we obtain \footnote{$\Vert f\Vert_2$ is the norm in the space $L^{2}(\mathbb{R})$ of square integrable functions w.r.t. the Lebesgue measure.}
\begin{equation*}
D\big(\widetilde{X_{N}}\big)=D\big(\widetilde{X_{N}}\!\parallel \!Z\big) \leq e^{-\frac{C^{2}%
}{2}}+\sqrt{2\pi }e^{\frac{C^{2}}{2}}\Vert \widetilde{p_{N}}-\phi \Vert
_{2}^{2},
\end{equation*}%
where $\phi$ is the standard normal density. 
Let $\Phi _{\widetilde{X_{N}}}$ and $\Phi $ be the characteristic functions
of random variables $\widetilde{X_{N}}$ and $Z$ respectively. The Plancherel
formula yields 
\begin{equation*}
\Vert \widetilde{p_{N}}-\phi \Vert _{2}^{2}=\frac{%
1}{2\pi }\Vert \Phi _{\widetilde{X_{N}}}-\Phi \Vert _{2}^{2}\leq \frac{1}{\pi }\left( \Vert \Phi _{\widetilde{X_{N}}}\Vert
_{2}^{2}+\Vert \Phi \Vert _{2}^{2}\right) .
\end{equation*}%
Let $\Phi _{\varepsilon }$ be the characteristic function of $\varepsilon
(B) $. We have 
\begin{equation*}
|\Phi _{\widetilde{X_{N}}}(\xi )|=\prod_{k=N+1}^{\infty }\Big\vert\Phi
_{\varepsilon }\left( \frac{a_{k}}{\sqrt{\sigma (X_{N})}}\xi \right) \!%
\Big\vert\leq \Big\vert\Phi _{\varepsilon }\left( \frac{a_{N+1}}{\sqrt{%
\sigma (X_{N})}}\xi \right) \!\Big\vert,
\end{equation*}%
whence 
\begin{align*}
\Vert \Phi _{\widetilde{X_{N}}}\Vert _{2}^{2}& \leq \int %
\Big\vert\Phi _{\varepsilon }\left( \frac{a_{N+1}}{\sqrt{\sigma (X_{N})}%
}\xi \right) \Big\vert^{2}\text{d}\xi 
 =\frac{\sqrt{\sigma(X_{N})}}{a_{N+1}}\Vert \Phi _{\varepsilon }\Vert
_{2}^{2}=\frac{\sqrt{\sum_{k>N}a_{k}^{2}}}{a_{N+1}}\Vert
\Phi _{\varepsilon }\Vert _{2}^{2} \\
&\leq C^{\prime} \Vert \Phi _{\varepsilon }\Vert _{2}^{2},
\end{align*}%
where $C^{\prime}$ does not depend on $N$.

Next we want to estimate the Lyapunov ratio 
\begin{equation*}
\mathbb{L}_{s}=\frac{\sum_{B\in \mathcal{B}(O)}\mathbb{E}|X_{B}|^{s}+\mathbb{%
E}|X_{N}|^{s}}{B_{N}^{s/2}}.
\end{equation*}%
Since $\mathcal{U}_{N}$ and $X_{N}$ are independent, we have
\begin{equation*}
B_{N}=\sigma \left( \overline{U}_{N}\right) ^{2}=\sigma \left( \mathcal{U}%
_{N}\right)^{2}+\sigma \left( X_{N}\right) ^{2}.
\end{equation*}
By \cite[Claim 1]{BGMS}, \footnote{$%
f(x)\asymp g(x)$ means that there are some constants $c,C>0$ such that $f(x)\leq cg(x)$
and $g(x)\leq Cf(x)$.} 
\begin{equation*}
\sigma \left( X_{N}\right) ^{2}\asymp \left( n_{0}\cdots n_{N}n_{N+1}\right)
^{-\delta }
\end{equation*}%
and 
\begin{equation}
\sigma \left( \mathcal{U}_{N}\right) ^{2}\asymp \left\{ 
\begin{array}{ccc}
\left( n_{0}\cdots n_{N}\right) ^{-1} & \text{if} & \delta >1, \\ 
N\cdot \left( n_{0}\cdots n_{N}\right) ^{-1} & \text{if} & \delta =1, \\ 
\left( n_{0}\cdots n_{N}\right) ^{-\delta } & \text{if} & \delta <1,%
\end{array}%
\right.  \label{U-L comparisong}
\end{equation}%
for $N$ large enough. Hence, for $\delta \geq 1$ and $N$ big enough, 
\begin{equation}
B_{N}\asymp \sigma \left( \mathcal{U}_{N}\right) ^{2}.  \label{B_N}
\end{equation}%
Further we have 
\begin{align*}
\sum_{B\in \mathcal{B}_{0}(O)}\mathbb{E}|X_{B}|^{s}+\mathbb{E}|X_{N}|^{s}&
=\left( \frac{a_{0}}{n_{1}\cdots n_{N}}\right) ^{s}\sum\limits_{B_{0}\in
H_{0}:B_{0}\subset O}\!\!\!\! \mathbb{E}|\varepsilon (B_{0})|^{s} 
+\left( \frac{a_{1}}{n_{2}\cdots n_{N}}\right)
^{s}\!\!\!\!\sum\limits_{B_{1}\in H_{1}:B_{1}\subset O}\!\!\!\! \mathbb{E}|\varepsilon
(B_{1})|^{s} \\
& \qquad +\ldots +a_{N}^{s}\mathbb{E}|\varepsilon (O_{N})|^{s}+\mathbb{E}%
\Big\vert\sum_{k=N+1}^{\infty }a_{k}\varepsilon (O_{k})\Big\vert^{s} \\
& \asymp \frac{1}{v_{N}^{s-1}}\left( 1+v_{1}^{s-1-s\delta /2}+\ldots
+v_{N}^{s-1-s\delta /2}\right)
\end{align*}%
whence 
\begin{equation*}
\sum_{B\in \mathcal{B}_{0}(O)}\mathbb{E}|X_{B}|^{s}+\mathbb{E}%
|X_{N}|^{s}\asymp 
\begin{cases}
\frac{1}{v_{N}^{s-1}} & \text{if $s-1-s\delta /2<0$,} \\ 
\frac{1}{v_{N}^{s/2-1}} & \text{if $s-1-s\delta /2\geq 0$}.%
\end{cases}%
\end{equation*}%
Combining this with \eqref{B_N} we obtain 
\begin{equation*}
\mathbb{L}_{s}\asymp 
\begin{cases}
\frac{1}{v_{N}^{s/2-1}} & \text{if $s-1-s\delta /2<0$,} \\ 
\frac{1}{v_{N}^{s/2(\delta -1)}} & \text{if $s-1-s\delta /2\geq 0$.}%
\end{cases}%
\end{equation*}%
Thus, for $s=3$ and $\delta >1$, we get the desired result.\newline
\textbf{The case} $(\delta =1)$: Let us introduce auxiliary notation 
\begin{align*}
Y_{0}& =\frac{a_{0}}{n_{1}...n_{N}}\sum\limits_{B\in H_{0}:B\subseteq
O}\varepsilon (B), \\
Y_{1}& =\frac{a_{1}}{n_{2}...n_{N}}\sum\limits_{B\in H_{1}:B\subseteq
O}\varepsilon (B), \\
& \ \ldots \  \\
Y_{N}& =a_{N}\varepsilon (O_{N})\quad \mathrm{and}\quad Y_{N+1}=X_{N}.
\end{align*}%
Clearly, each of $Y_{k}$ has mean zero and 
\begin{equation*}
\Lambda _{N}=\frac{Y_{0}+Y_{1}+\ldots +Y_{N+1}}{B_{N+1}^{1/2}},\quad
B_{N+1}=\sum_{k=0}^{N+1}\sigma \left( Y_{k}\right) ^{2}.
\end{equation*}%
We again apply Theorem \ref{THM_Berry-Essen} because, by Lemma \ref%
{lemma_entropy}, 
\begin{equation*}
D(Y_{k})=D\left( \sum\limits_{B\in H_{k}:B\subseteq O}\varepsilon (B)\right)
\leq D(\epsilon (B))<\infty ,\quad 0\leq k\leq N.
\end{equation*}%

To estimate the Lyapunov ratios we use the Marcinkiewicz-Zygmund inequality 
\cite[Chapter VII, \S 3]{Shiryaev}: for independent random variables $\xi
_{1},\xi _{2},\ldots ,\xi _{n}$ with mean zero and for all $p>1$ there is a
constant $C(p)>0$ such that 
\begin{equation*}
\mathbb{E}\Big\vert\sum_{k=1}^{n}\xi _{k}\Big\vert^{p}\leq C(p)\,\mathbb{E}%
\left( \sum_{k=1}^{n}\xi _{k}^{2}\right) ^{p/2}.
\end{equation*}%
Setting 
\begin{equation*}
\xi _{k}=\sum_{B\in H_{k}:B\subset O}\varepsilon (B),
\end{equation*}%
the Marcinkiewicz-Zygmund inequality yields 
\begin{equation*}
\mathbb{E}|Y_{k}|^{s}=\left( \frac{a_{k}}{n_{k+1}\cdots n_{N}}\right) ^{s}%
\mathbb{E}|\xi _{k}|^{s}\leq C(s)\mathrm{Var}(\epsilon (B))\frac{a_{k}^{s}}{%
(n_{k+1}\cdots n_{N})^{s/2}}.
\end{equation*}%
It follows, see \eqref{C-B condition11}, that 
\begin{align*}
\sum_{k=0}^{N}\mathbb{E}|Y_{k}|^{s}& \leq C_1(s) \left( \frac{a_{0}^{s}}{%
(n_{1}\cdots n_{N})^{s/2}}+\frac{a_{1}^{s}}{(n_{2}\cdots n_{N})^{s/2}}%
+\ldots +a_{N}^{s}\right) \\
& = \frac{C_1(s)}{v_{N}^{s/2}}\left( a_0^s+a_1^sn_{1}^{s/2}+\ldots
+a_N^sv_{N}^{s/2}\right) \leq C_2(s) \frac{N}{v_{N}^{s/2}}.
\end{align*}%
Hence, using \eqref{B_N} and \eqref{U-L comparisong} with $\delta =1$, we
finally get 
\begin{equation*}
\mathbb{L}_s\leq \frac{C}{N^{s/2-1}}.
\end{equation*}%
By Theorem \ref{THM_Berry-Essen}, the result follows.
\end{proof}

\begin{remark}
Applying similar reasoning we estimate the relative entropy distance. For $%
O\in H_N$, 
\begin{align*}
D(\Lambda_O \!\parallel \! \mathcal{N}) \leq 
\begin{cases}
C\, N^{-1} & \text{for $\delta =1,$} \\ 
C\, v_N^{-2\min\{(\delta -1),\frac{1}{2} \}} & \text{for $\delta >1$},%
\end{cases}%
\end{align*}
for some $C>0$ and all $N$. 
\end{remark}

\section{An example}

As an example we consider the space $X=\mathbb{Q}_{p}$ equipped with its
standard ultrametric $\left\vert x-y\right\vert _{p}$ and its normalized Haar
measure. Let $\mathcal{D}^{\alpha }$, $\alpha >0$, be a homogeneous
hierarchical Laplacian uniquely defined by its eigenvalues 
\begin{equation*}
\lambda {^{\alpha }}(B)=\left( \frac{p}{\mathrm{diam}(B)}\right) ^{\alpha },%
\text{ \ }B\in \mathcal{B}.
\end{equation*}%
Let $\mathcal{D}^{\alpha }\left( \omega \right) $ be its random perturbation
by i.i.d. $\{\varepsilon (B)\}_{B\in \mathcal{B}}$ as defined in Section \ref{Prem_sec}.
As in the previous section we assume that $\varepsilon (B)$ admits a bounded density.
We notice that $\mathcal{D}^{\alpha }$ satisfies condition (\ref{lambda-B
condition}) with $\delta =2\alpha .$ In particular, for any $\alpha \geq 1/2$
the normalized arithmetic means $\Lambda _{O}^{\alpha }\left( \omega \right) 
$ converge as $O\rightarrow \varpi $ to the standard normal random variable $%
Z$ in the sense of the total variation distance. In this section we study
convergence assuming that $0<\alpha <1/2.$

\begin{theorem}
\label{Non-normal}For any $0<\alpha <1/2$ there exists a non-gaussian random variable $%
\Lambda ^{\alpha }$ such that the normalized arithmetic means $\Lambda _{O}^{\alpha }$ converge to $\Lambda
^{\alpha }$ in the sense of the total variation distance. 
\end{theorem}

\begin{proof}
Following line-by-line the proof of the Theorem \ref{Thm_approx} we write 
\begin{equation*}
\sigma \left( \overline{U}_{N}\right) ^{2}=\sigma \left( \mathcal{U}%
_{N}\right) ^{2}+\sigma \left( X_{N}\right) ^{2}.
\end{equation*}%
As $\lambda _{k}=p^{-\alpha (k-1)}$ and $c_{k}=\left( p^{\alpha }-1\right)
p^{-\alpha k}$, we get
\begin{equation*}
a_{k}=c_{k}/\lambda _{0}=\left( p^{\alpha }-1\right) p^{-\alpha (k+1)},\text{
}k\geq 0.
\end{equation*}%
Using the above data and setting $\sigma^2 = \mathrm{Var} \left( \varepsilon \right)$ we estimate $\sigma \left( \mathcal{U}_{N}\right)$ and $%
\sigma \left( X_{N}\right)$ at $\infty $ as follows
\begin{eqnarray}
\sigma \left( \mathcal{U}_{N}\right) ^{2} &=&\sigma ^{2}\left( \frac{a_{0}^{2}}{%
p^{N}}+\frac{a_{1}^{2}}{p^{N-1}}+...+a_{N}^{2}\right) 
= \sigma ^{2}\left( p^{\alpha }-1\right) ^{2}\sum\limits_{0\leq l\leq
N}p^{-N+l-2\alpha (l+1)}  \notag \\
&\sim &\frac{\sigma ^{2}\left( p^{\alpha }-1\right) ^{2}}{1-p^{2\alpha -1}}%
p^{-2\alpha \left( N-1\right) }=\frac{\sigma ^{2}}{1-p^{2\alpha -1}}a_{N}^{2}
 \label{ass-sigma-U}
\end{eqnarray}%
and%
\begin{eqnarray}
\sigma \left( X_{N}\right) ^{2} &=&\sigma ^{2}\left(
a_{N+1}^{2}+a_{N+2}^{2}+...\right) 
=\sigma ^{2}\left( p^{\alpha }-1\right) ^{2}\sum\limits_{l\geq
N+1}p^{-2\alpha (l+1)}   \notag \\
&\sim &\frac{\sigma ^{2}\left( p^{\alpha }-1\right) ^{2}}{1-p^{-2\alpha }}%
p^{-2\alpha (N+2)}=\frac{\sigma ^{2}}{1-p^{-2\alpha }}a_{N+1}^{2}.  \label{ass-sigma-V}
\end{eqnarray}%
Let $\{\varepsilon _{i}\}_{i\geq 0}$ be i.i.d. random variables independent
of $\{\varepsilon (B)\}_{B\in \mathcal{B}}$ and having the same common
distribution as $\{\varepsilon (B)\}_{B\in \mathcal{B}}.$ By (\ref{V-L equation})
and (\ref{ass-sigma-V}) the random variable $X%
_{N}/\sigma \left( X_{N}\right)$ converges in law to the random variable%
\begin{equation*}
X\mathfrak{=}\sqrt{1-p^{-2\alpha }}\left( \frac{\varepsilon _{0}}{\sigma %
\left( \varepsilon _{0}\right) }+p^{-\alpha }\frac{\varepsilon _{1}}{\sigma %
\left( \varepsilon _{2}\right) }+...+p^{-k\alpha }\frac{\varepsilon _{k}}{%
\sigma \left( \varepsilon _{k}\right) }+...\right) .
\end{equation*}%

Let $\{\varepsilon _{ij}\}_{i,j\geq 0}$ be i.i.d. random variables
independent of both $\{\varepsilon _{i}\}_{i\geq 0}$ and $\{\varepsilon
(B)\}_{B\in \mathcal{B}}$ and having the same common distribution as $%
\{\varepsilon (B)\}_{B\in \mathcal{B}}.$ Define the random variables%
\begin{equation*}
S_{k}=\sum\limits_{0\leq j\leq p^{k}}\varepsilon _{kj},\text{ }k=0,1,2,...%
\text{ .}
\end{equation*}%
By (\ref{U-L equation}) and (\ref{ass-sigma-U}) the random variable $%
\mathcal{U}_{N}/\sigma \left( \mathcal{U}_{N}\right)$ converges in law to the
random variable%
\begin{equation*}
U=\sqrt{1-p^{2\alpha -1}}\sum\limits_{k\geq 0}p^{\left( 2\alpha -1\right)
k/2}\frac{S_{k}}{\sigma \left(S_{k}\right) }.
\end{equation*}%
Finally, the random variable%
\begin{equation*}
\Lambda _{N}=\frac{\overline{U}_{N}}{\sigma \left( \overline{U}_{N}\right) }=%
\frac{\sigma \left( \mathcal{U}_{N}\right)}{\sigma \left( \overline{U}_{N}\right) 
}\frac{\mathcal{U}_{N}}{\sigma \left( \mathcal{U}_{N}\right) }+\frac{\sigma
\left( X_{N}\right)}{\sigma \left( \overline{U}_{N}\right) }\frac{%
X_{N}}{\sigma \left( X_{N}\right)}
\end{equation*}%
converges in law to the random variable%
\begin{equation*}
\Lambda =\sqrt{\frac{1-p^{-2\alpha }}{1-p^{-1}}}U+\sqrt{\frac{p^{-2\alpha
}-p^{-1}}{1-p^{-1}}}V\mathfrak{.}
\end{equation*}%
By Cram\'{e}r's theorem $U$ and $V$, and therefore $\Lambda $ are not Gaussian. 

We claim that $\Lambda_N$ converges to $\Lambda$ in the total variation distance. 
To prove the claim we work with characteristic functions. As before for a random variable $X$ we denote by $\Phi _{X}(\xi )=\mathbb{E(}\exp (i\xi X)\mathbb{)}$ its characteristic function.
The following inequality holds
\begin{align}\label{ineq111}
\Big\vert \Phi_{\Lambda_N}\left(\frac{\sigma \left( \overline{U}_N\right)}{a_{N+1}}\xi \right)\Big\vert \leq 
\big\vert \Phi_{\varepsilon}\left( \xi \right) \big\vert
\big\vert \Phi_{\varepsilon}\left( p^{-\alpha}\xi \right) \big\vert .
\end{align}
Equations \eqref{ass-sigma-U} and \eqref{ass-sigma-V} yield
\begin{align*}
A_N = \frac{\sigma \left(\overline{U}_N\right) }{a_{N+1}}\to \left( \frac{\sigma ^2(1-p^{-1})}{(p^{-2\alpha}-p^{-1})(1-p^{-2\alpha})}\right) ^{1/2} := A.
\end{align*}
Moreover $\Phi_{\Lambda_N}(A_N \xi) \to \Phi_{\Lambda}(A \xi)$ pointwise.
As $\varepsilon$ has a bounded density, $\Phi_\varepsilon \in L^2(\mathbb{R})$ and therefore the function in the right-hand side of \eqref{ineq111} is in $L^1(\mathbb{R})$. This in turn implies that 
$\Phi_{\Lambda_N}(\xi) \to \Phi_\Lambda (\xi)$ in $L^1(\mathbb{R})$. It follows that the density $p_N$ of $\Lambda_N$ converges pointwise to the density $p_\Lambda$ of $\Lambda$. Finally Scheff\'{e}'s lemma \cite[Section 2.9]{VanDerVart} yields that $\Lambda_N \to \Lambda$ in the total variation distance, as desired.
\end{proof}

\section*{Acknowledgment}
We wish to thank G.P. Chistyakov and F. G\"{o}tze for fruitful discussions.

\bigskip

\begin{tabular}{ll}
\textsc{Alexander Bendikov}\textsc{\ \ \ \ \ \ \ \ \ \ } & \textsc{Wojciech Cygan} \\ 
Institute of Mathematics & Institute of Mathematics \\ 
University of Wroclaw, Poland & University of Wroclaw, Poland \\ 
bendikov@math.uni.wroc.pl & wojciech.cygan@uwr.edu.pl \\ 
\  & 
\end{tabular}

\end{document}